\documentclass[review]{elsarticle}

\usepackage{lineno,hyperref}
\usepackage{epstopdf}
\usepackage{subfigure}
\usepackage{color}
\usepackage{mathdots}
\usepackage{enumerate}
\usepackage{tikz}
\usepackage{amsmath}

\usetikzlibrary{positioning,chains,fit,shapes,calc}
\newtheorem{theorem}{Theorem}[section]

\newtheorem{lemma}[theorem]{Lemma}

\newtheorem{proof}[theorem]{Proof}

\newtheorem{example}{Example}

\journal{Linear Algebra and Its Applications}

\begin{document}

\begin{frontmatter}

\title{A Short Note on Multilevel Toeplitz Matrices}
\tnotetext[mytitlenote]{Fully documented templates are available in the elsarticle package on \href{http://www.ctan.org/tex-archive/macros/latex/contrib/elsarticle}{CTAN}.}

\author{Lei Cao \fnref{myfootnote}}
\address{Department of Mathematics, Halmos College, Nova Southeastern University, FL 33314}

\author{Selcuk Koyuncu \corref{mycorrespondingauthor}}
\cortext[mycorrespondingauthor]{Corresponding author}
\address{Department of Mathematics, University of North Georgia, GA 30566}
\ead{skoyuncu@ung.edu}

\begin{abstract}
Chien, Liu, Nakazato and Tam proved that all $n\times n$ classical Toeplitz matrices (one-level Toeplitz matrices) are unitarily similar to complex symmetric matrices via two types of unitary matrices and the type of the unitary matrices only depends on the parity of $n.$  In this paper we extend their result to multilevel Toeplitz matrices that any multilevel Toeplitz matrix is unitarily similar to a complex symmetric matrix. We provide a method to construct the unitary matrices that uniformly turn any multilevel Toeplitz matrix to a complex symmetric matrix by taking tensor products of these two types of unitary matrices for one-level Toeplitz matrices according to the parity of each level of the multilevel Toeplitz matrices. In addition, we introduce a class of complex symmetric matrices that are unitarily similar to some $p$-level Toeplitz matrices.
\end{abstract}


\end{frontmatter}


\section{Introduction}

Although every complex square matrix is unitarily similar to a complex symmetric matrix (see Theorem 4.4.24, \cite{HJ}), it is known that not every $n\times n$ matrix is unitarily similar to a complex symmetric matrix when $n\geq 3$ (See \cite{GP}). Some characterizations of matrices unitarily equivalent to a complex symmetric matrix (UECSM) were given by \cite{BS} and \cite{GPW}. Very recently, a constructive proof that every Toeplitz matrix is unitarily similar to a complex symmetric matrix was given in \cite{TamLiu} in which the unitary matrices turning all $n\times n$ Toeplitz matrices to complex symmetric matrices was given explicitly. An interesting fact was that the unitary matrices only depend on the parity of the size.

Multilevel Toeplitz matrices arise naturally in multidimensional Fourier analysis when a periodic multivariable real function is considered \cite{TZ}. In this paper, we show that any multilevel Toeplitz matrix is unitarily similar to a complex symmetric matrix. Along the line in \cite{TamLiu}, a constructive proof is given. One can take tensor product of the unitary matrices defined in \cite{TamLiu} and identity matrices appropriately to construct the unitary matrix turning any multilevel Toeplitz matrix to a complex symmetric matrix which only depends on the parity of the size of each level. In section 4, we provide two examples of constructing the unitary transition matrices of a 2-level Toeplitz matrix and a 3-level Toeplitz matrix to illustrate our main results in section 3. The converse is considered in Section 5, in which we give the necessary and sufficient condition for a $2^p \times 2^p$ complex symmetric matrix similar to a $p$-level Toeplitz matrix under the unitary transformation given in Section 3.

\section{Preliminary and Notations}

A classical 1-level matrix $T_n \in {\mathbf{C}}^{n\times n}$ is called Toeplitz if
it  has constant entries along its diagonals, i.e, if it is of the
form
$$T_n=\begin{bmatrix} t_0 & t_{-1}& \cdots& t_{-n+1} \\
t_1 & t_0 & \ddots&t_{-n+2} \\
\vdots&\ddots&\ddots&\vdots \\
t_{-1+n}& \ldots& t_1 & t_0 \end{bmatrix}.$$

A $p$-level Toeplitz matrix, denoted by $T^{(p)},$ has Toeplitz structure on each level and
corresponds to a $p$-variate generating function.

 For an integer $p\geq 1,$ a $p$-level Toeplitz matrix of size $(n_0 n_1 n_2  n_3 \cdots n_p) \times (n_0n_1 n_2\cdots n_p)$ where $n_0=1$ and $n_i\in \mathbf{N}$ for $i=1,2,\ldots, p,$ is a block
Toeplitz
matrix of the form $$T^{(p)}=\begin{bmatrix} T_0^{(p-1)} & T_{-1}^{(p-1)}& \cdots& T_{-n_p+1}^{(p-1)} \\
T_1^{(p-1)} & T_0^{(p-1)} & \ddots&T_{-n_p+2}^{(p-1)} \\
\vdots&\ddots&\ddots&\vdots \\
T_{-1+n_p}^{(p-1)}& \ldots& T_1^{(p-1)} & T_0^{(p-1)} \end{bmatrix},$$ where each block $T_l^{(p-1)}$
is itself a $(p-1)$-level Toeplitz matrix of size $n_1 \cdot n_2  \cdots  n_{p-1} .$
For instance if p=2, we have the following two-level Toeplitz matrix with Toeplitz blocks
\begin{equation}T^{(2)}=\begin{bmatrix}T_0^{(1)} & T_{-1}^{(1)} \\ T_{1}^{(1)} & T_{0}^{(1)}\end{bmatrix}=\begin{bmatrix}

 t_{0,0}  &   t_{0,-1}  &    t_{-1,0}     &  t_{-1,-1} \\
 t_{0,1}  &   t_{0,0}   &    t_{-1,1}     &  t_{-1,0} \\
 t_{1,0}  &   t_{1,-1}  &    t_{0,0}      &  t_{0,-1} \\
 t_{1,1}  &   t_{1,0}   &    t_{0,1}      &  t_{0,0}
\end{bmatrix},\end{equation} where $T_0^{(1)}=\begin{bmatrix}
 t_{0,0}  &   t_{0,-1} \\
 t_{0,1}  &   t_{0,0} \end{bmatrix},T_{-1}^{(1)}=\begin{bmatrix}t_{-1,0}     &  t_{-1,-1} \\t_{-1,1}     &  t_{-1,0} \end{bmatrix}$ and $ T_1^{(1)}=\begin{bmatrix}
 t_{1,0}  &   t_{1,-1} \\
 t_{1,1}  &   t_{1,0} \end{bmatrix}$ are classical $1$-level Toeplitz matrices.

More generally, let $p\in \mathbf{N}.$ For $0\leq i\leq p,$ let $n_i\in \mathbf{N}$ with $n_0=1.$ Denote $\displaystyle s_k=\prod_{i=0}^k n_i$ for $k=1,2,\ldots, p.$ Denote
$$T^{(0)}_{i-j}=t_{i-j} \ \ \  {\rm for} \ \  |i-j|\leq n_1-1, $$ where $t_{i-j} \in \mathbf{C}.$ Then a $p$-level Toeplitz matrix, $T^{(p)}$ is of size $s_p$ and denoted by
$$T^{(p)}=\begin{bmatrix} T_0^{(p-1)} & T_{-1}^{(p-1)}& \cdots& T_{-n_p+1}^{(p-1)} \\
T_1^{(p-1)} & T_0^{(p-1)} & \ddots&T_{-n_p+2}^{(p-1)} \\
\vdots&\ddots&\ddots&\vdots \\
T_{-1+n_p}^{(p-1)}& \ldots& T_1^{(p-1)} & T_0^{(p-1)} \end{bmatrix},$$ where the $(i,j)$th block of $T^{(p)}$
is the $(p-1)$-level Toeplitz matrix, $T^{(p-1)}_{i-j},$ of size $s_{p-1}$ for $|i-j|\leq n_p-1.$
Note that $1$-level Toeplitz matrix $T^{(1)}$ is a regular Toeplitz matrix. Using the notation of $p$-level Toeplitz matrices, the main result in \cite{TamLiu} is stated as the following theorems.

\begin{theorem}\label{TL}(Theorem 3.3 \cite{TamLiu})
Every $1$-level Toeplitz matrix $T\in \mathbf{C}_{n\times n}$ is unitarily similar to
a symmetric matrix. Moreover, the following $n$ by $n$ even and odd unitary matrices  uniformly turn
all Toeplitz matrices with even sizes and odd sizes into symmetric matrices respectively via similarity:
\begin{enumerate}[(i)]
\item when $n=2m$ with $m\geq 1,$
\begin{equation}\label{Ue}U(n)=\frac{1}{\sqrt{2}}\begin{pmatrix} 1 & && & & i \\ & \ddots && & \iddots & \\ && 1 & i &&\\ && 1 & -i &&\\& \iddots && & \ddots & \\  1 & && & & -i  \end{pmatrix}\end{equation}

\item when $n=2m+1$ with $m\geq 1,$
\begin{equation}\label{Uo}U(n)=\frac{1}{\sqrt{2}}\begin{pmatrix} 1 & && && & i \\ & \ddots &&& & \iddots & \\ && 1 &0& i &&\\ &&0&\sqrt{2}&0&&\\&& 1 &0& -i &&\\& \iddots &&& & \ddots & \\  1 & &&& & & -i  \end{pmatrix}.\end{equation}

\end{enumerate}

\end{theorem}

Let $J_n$ be the $n\times n$ matrix with all elements zero except the elements on the anti diagonal which are all $1'$s. That is, $$J_n=\begin{pmatrix}&&&&1 \\ &&&1& \\ && \iddots &&\\ & 1 &&&\\ 1&&&& \end{pmatrix}.$$ Then a Topeplitz matrix with any size can be unitarily turned into a symmetric matrix by the matrix $$U = \frac{1}{\sqrt{2}}(I_n+iJ_n)$$ which is clearly unitary.

\begin{theorem}(Theorem 3.1 \cite{TamLiu}) \label{thm2} Every $n\times n$ Toeplitz matrix $T=(t_{ij})$ is unitarily similar to a symmetric matrix $B=(b_{ij})$ via the unitary matrix $$U = \frac{1}{\sqrt{2}}(I_n+iJ_n).$$ More specifically,
$$b_{ij}=\frac{1}{2}(t_{i-j}+t_{j-i})+\frac{i}{2}(t_{i+j-n-1}-t_{n+1-i-j}).$$

\end{theorem}

\section{Multilevel Unitary Symmetrization}

Denote $U(n)$ an $n\times n$ unitary matrix and if $n$ is even, $U(n)$ is defined by (\ref{Ue}); if $n$ is odd, $U(n)$ is defined by (\ref{Uo}).

 \begin{theorem} \label{mainThm} Let $T^{(p)}$ be a $p$-level Toeplitz matrix of size $s_p$. Then there exists a unitary matrix $U$ of size $s_p,$ such that $$U^*T^{(p)}U$$ is symmetric and the unitary transition matrix $U$ is   $$U=U_1\cdots U_{p-1}U_p,$$ where

 $$U_i=I_{n_p}\otimes I_{n_{p-1}}\otimes \ldots \otimes I_{n_{i+1}} \otimes U(n_i) \otimes I_{n_{i-1}}\otimes \ldots \otimes I_{n_2} \otimes I_{n_1}$$ for $i=1,2,3,\ldots,p.$
\end{theorem}

\begin{proof}
 We prove it by mathematical induction on $p$.

 For $p=1,$ it is true due to Theorem 3.3 in \cite{TamLiu}.

 Assume the result is true for $k$ meaning that there exists a unitary matrix $\tilde{U}$ of size $s_{k} \times s_{k}$ such that $$\tilde{U}^*T^{(k)}\tilde{U}$$ is symmetric for any $k$-level Toeplitz matrix with size $s_k.$

  That is, any $k$-level Toeplitz matrix $T^{(k)}$ is unitarily similar to a symmetric matrix via $\tilde{U}=U_1 \cdots U_k.$ This implies the following $$U_k^*U_{k-1}^*\cdots U_1^*T^{(k)}U_1U_2\cdots U_k=\tilde{U}^*T^{(k)}\tilde{U}$$ is symmetric.

Let us prove the result for case $p=k+1.$

Consider a $(k+1)$-level Toeplitz matrix $T^{(k+1)}$ with size $s_{k+1}$
$$T^{(k+1)}=\begin{pmatrix} T_0^{(k)} & T_{-1}^{(k)}& \cdots& T_{-n_{k+1}+1}^{(k)} \\
T_1^{(k)} & T_0^{(k)} & \cdots&T_{-n_{k+1}+2}^{(k)} \\
\vdots&\vdots&\ddots&\vdots \\
T_{-2+n_{k+1}}^{(k)}& T_{-3+n_{k+1}}^{(k)} & \ldots & T_{-1}^{(k)}\\
T_{-1+n_{k+1}}^{(k)}& T_{-2+n_{k+1}}^{(k)} & \ldots & T_0^{(k)} \end{pmatrix}.$$
where all blocks $T_{i-j}^{(k)}$ are $k$-level Toeplitz matrices of  size $s_k \times s_k$ and note that $s_k= \displaystyle \prod_{i=0}^{k} n_i.$
 Next we define $$\hat{U}=I_{n_{k+1}}\otimes\tilde{U}=\begin{pmatrix} \tilde{U} & 0 & \ldots & 0\\ 0&  \tilde{U}& \ldots & 0  \\ \vdots&\vdots&\ddots&\vdots  \\0 & 0 &\ldots & \tilde{U} \end{pmatrix}.$$ 

Let $\tilde{S}= \hat{U}^*T^{(k+1)}\hat{U}.$ Then
$$\tilde{S}=\begin{pmatrix}  \tilde{U}^*T_0^{(k)}\tilde{U} & \tilde{U}^*T_{-1}^{(k)}\tilde{U}& \cdots& \tilde{U}^*T_{-n_{k+1}+1}^{(k)}\tilde{U} \\
\tilde{U}^*T_1^{(k)}\tilde{U} & \tilde{U}^*T_0^{(k)}\tilde{U} & \cdots& \tilde{U}^*T_{-n_{k+1}+2}^{(k)}\tilde{U} \\
\vdots&\vdots&\ddots&\vdots \\
\tilde{U}^*T_{-1+n_{k+1}}^{(k)}\tilde{U}& \tilde{U}^*T_{-2+n_{k+1}}^{(k)}\tilde{U} & \ldots & \tilde{U}^*T_0^{(k)}\tilde{U} \end{pmatrix}$$
By induction hypothesises, $\tilde{U}^*T_{i-j}^{(k)}\tilde{U}$ is symmetric.  Denote  $\tilde{U}^*T_{i-j}^{(k)}\tilde{U}$ by $\tilde{S}_{i-j},$ then
$$ \tilde{S}=\begin{pmatrix} \tilde{S}_0 & \tilde{S}_{-1}& \cdots& \tilde{S}_{-n_{k+1}+1}\\
\tilde{S}_1 & \tilde{S}_0 & \cdots&\tilde{S}_{-n_{k+1}+2} \\
\vdots&\vdots&\ddots&\vdots \\
\tilde{S}_{-1+n_{k+1}}& \tilde{S}_{-2_{k+1}+n} & \ldots & \tilde{S}_0 \end{pmatrix},$$ where $\tilde{S}_{t}$ is symmetric for $t=-n_{k+1}+1,-n_{k+1}+2,\ldots,-1,0,1,\ldots, n_{k+1}-1.$

 Let $$U_{k+1}=U(n_{k+1})\otimes I_{s_k},$$

that is, $$U_{k+1}=\frac{1}{\sqrt{2}}\begin{pmatrix} I_{s_k} & && & & iI_{s_k} \\ & \ddots && & \iddots & \\ && I_{s_k} & iI_{s_k} &&\\ && I_{s_k} & -iI_{s_k} &&\\& \iddots && & \ddots & \\  I_{s_k} & && & & -iI_{s_k}  \end{pmatrix}$$
if $n_{k+1}$ is even;
$$U_{k+1}=\frac{1}{\sqrt{2}}\begin{pmatrix} I_{s_k} & && && & iI_{s_k} \\ & \ddots &&& & \iddots & \\ && I_{s_k} &0& iI_{s_k} &&\\ &&0&\sqrt{2}I_{s_k}&0&&\\&& I_{s_k} &0& -iI_{s_k} &&\\& \iddots &&& & \ddots & \\  I_{s_k} & &&& & & -iI_{s_k}  \end{pmatrix}$$
if $n_{k+1}$ is odd.

It suffices to show that $U_{k+1}^*\tilde{S}U_{k+1}$ is symmetric. Let $ V=U_{k+1}.$
\begin{itemize}

\item Suppose $n_{k+1}$ is even, that is $n_{k+1}=2t$ for some integer $t.$ Then

$$\sqrt{2}V_{ij}=\left\{
\begin{aligned}
I_m \ \ \ \ &1\leq i \leq t \ {\rm  and} \  i=j\\
-iI_m \ \ \ \ &t+1\leq i \leq n_{k+1} \ {\rm  and} \  i=j  \\
iI_m \ \ \ \ &1\leq i \leq t \ {\rm  and} \  i+j=n_{k+1}+1  \\
I_m \ \ \ \ &t+1\leq i \leq n_{k+1} \ {\rm  and} \  i+j=n_{k+1}+1  \\
0 \ \ \ \ &{\rm otherwise}
\end{aligned}
\right.$$

and

$$\sqrt{2}V^*_{ij}=\left\{
\begin{aligned}
I_m \ \ \ \ &1\leq i \leq t \ {\rm  and} \  i=j\\
iI_m \ \ \ \ &t+1\leq i \leq n_{k+1} \ {\rm  and} \  i=j  \\
I_m \ \ \ \ &1\leq i \leq t \ {\rm  and} \  i+j=n_{k+1}+1  \\
-iI_m \ \ \ \ &t+1\leq i \leq n_{k+1} \ {\rm  and} \  i+j=n_{k+1}+1  \\
0 \ \ \ \ &{\rm otherwise}
\end{aligned}
\right.$$

which gives us that

$$(V^*\tilde{S})_{ij}=\left\{
\begin{aligned}
S_{i-j} +S_{n_{k+1}-j-(i+1)} \ \ \ \ &1\leq i \leq t \\
-i(S_{n_{k+1}-j-(i-1)}+iS_{i-j}) \ \ \ \ &t+1\leq i \leq n_{k+1}
\end{aligned}
\right.$$

Denote $$S=V^*\tilde{S} V=\begin{pmatrix} S_0 & S_{-1}& \cdots& S_{-n_{k+1}+1}\\
S_1 & S_0 & \cdots&S_{-n_{k+1}+2} \\
\vdots&\vdots&\ddots&\vdots \\
S_{-1+n_{k+1}}& S_{-2_{k+1}+n} & \ldots & S_0 \end{pmatrix}=\begin{pmatrix}Z_1 & Z_2 \\ Z_3 & Z_4 \end{pmatrix}$$ where $Z_1,Z_2, Z_3$ and $Z_4$ have the same size  and let $1\leq p,q\leq n_{k+1} $ be the indices. Then we get,
\begin{enumerate}[(i)]
\item For $1\leq p \leq t $ and $  1\leq q \leq t,$
\begin{equation}\label{eq1}
S_{pq}=\tilde{S}_{-(p-q)}+\tilde{S}_{p-q}+\tilde{S}_{-(2t-p-q+1)}+\tilde{S}_{2t-p-q+1}
\end{equation}

\item For $1\leq p \leq t$ and  $t+1\leq q \leq 2t$
\begin{equation}\label{eq2}
S_{pq}=i(\tilde{S}_{-(2t-p-q+1)}+\tilde{S}_{q-p})-i(\tilde{S}_{-(q-p)}+\tilde{S}_{2t-p-q+1})
\end{equation}

\item For $t+1\leq p \leq 2t$ and  $1\leq q \leq t,$
\begin{equation}\label{eq3}
S_{pq}= -i(\tilde{S}_{2t-p-q+1}+\tilde{S}_{-(p-q)})+i(\tilde{S}_{p-q}+\tilde{S}_{-(2t-p-q+1)})
\end{equation}

\item For $t+1\leq p \leq 2t$ and  $t+1\leq q \leq 2t,$
\begin{equation}\label{eq4}
S_{pq}= \tilde{S}_{q-p}+\tilde{S}_{p+q-1-2t}+\tilde{S}_{-(q-p)}+\tilde{S}_{-(2t+p+q+1)}
\end{equation}

\end{enumerate}

First note that \eqref{eq1} and \eqref{eq4} are the same due to the Toeplitz structure of $S.$  If we switch $p$ and $q$ in \eqref{eq1} or \eqref{eq4}, we have
\begin{eqnarray} \nonumber &&\tilde{S}_{-(q-p)}+\tilde{S}_{q-p}+\tilde{S}_{-(2t-q-p+1)}+\tilde{S}_{2t-q-p+1}\\
\nonumber &=& \tilde{S}_{p-q}+\tilde{S}_{-(p-q)}+\tilde{S}_{-(2t-q-p+1)}+\tilde{S}_{(2t-p-q+1)}
\end{eqnarray}
  which is equal to  \eqref{eq1} and \eqref{eq4} meaning that both $Z_1$ and $Z_4$ are symmetric. If we switch $p$ and $q$ in \eqref{eq2}, we have
  $$ i(\tilde{S}_{-(2t-q-p+1)}+\tilde{S}_{p-q})-i(\tilde{S}_{-(p-q)}+\tilde{S}_{2t-q-p+1})$$
  equal to \eqref{eq3} which shows that $Z_2=Z_3^t$ and $Z_2^t=Z_3.$  Hence $$S^t=\begin{pmatrix}Z_1^t & Z_3^t \\ Z_2^t & Z_4^t\end{pmatrix}=\begin{pmatrix}Z_1 & Z_2 \\ Z_3 & Z_4\end{pmatrix}=S.$$ Thus $S$ is symmetric.

\medskip

\item Suppose $n_{k+1}$ is odd. Then we can write $n_{k+1}=2t+1$ for some integer $t.$ Let $S=V^*\tilde{S}V.$ Similarly to the case for even, one can show $S_{pq}=S_{qp}$ for $p=1,2,\ldots,t,t+2,\ldots,2t+1$ and   $q=1,2,\ldots,t,t+2,\ldots,2t+1.$ In addition, straightforward calculation yields the $(t+1)$th row and the $(t+1)$th column as follows

    $$S_{pq}=\left\{
\begin{aligned}
\frac{\sqrt{2}}{2}(\tilde{S}_{t+1-p}+\tilde{S}_{p-t-1}) \ \ \ \  & 1\leq p \leq t\  {\rm  and}\ q=t+1 \\
\frac{\sqrt{2}}{2}(\tilde{S}_{q-t-1}+\tilde{S}_{t+1-q}) \ \ \ \  & p = t+1\  {\rm  and}\ 1\leq q \leq t \\
\tilde{S_0} \ \ \ \ & p=t+1\  {\rm  and}\ q=t+1\\
\frac{\sqrt{2}}{2}i(\tilde{S}_{q-t-1}+\tilde{S}_{t+1-q}) \ \ \ \  &  p=t+1\  {\rm  and}\ t+2 \leq q\leq 2t+1 \\
\frac{\sqrt{2}}{2}i(\tilde{S}_{t+1-p}+\tilde{S}_{p-t-1}) \ \ \ \  & t+2\leq p \leq 2t=1\  {\rm  and}\ q=t+1
\end{aligned}
\right.$$

Hence $S$ is symmetric.

\end{itemize}

\end{proof}

We also generalize Theorem \ref{thm2}, in which one does not need to consider the parity of the size. We denote $$V(n)=\frac{1}{\sqrt{2}}(I_n+iJ_n).$$

\begin{theorem} \label{Thm3.2}Let $T^{(p)}$ be a $p$-level Toeplitz matrix of size $s_p=\prod_{i=1}^p n_i.$ Then there exists a unitary matrix $V$ such that $V^*T^{(p)}V$ is symmetric, where $$V=\prod_{i=1}^n V_i$$ and $$V_i=I_{n_p}\otimes I_{n_{p-1}}\otimes \ldots \otimes I_{n_{i+1}} \otimes V(n_i) \otimes I_{n_{i-1}}\otimes \ldots \otimes I_{n_2} \otimes I_{n_1}$$ for $i=1,2,\ldots, p.$

\end{theorem}
\begin{proof} The proof will be omitted since it is similar to Theorem 3.1.

\end{proof}

\section{Examples}

Here are two examples to illustrate the constructions of the transition matrices given by Theorem \ref{mainThm} and Theorem \ref{Thm3.2} respectively .

\begin{example} Let \begin{align*}
 T=\left(\begin{array}{ccc|ccc}
i& 1& 0& 4& i& 1 \\
2 & i&1&i&4&i \\
3&2&i&1&i&4 \\ \hline
5&2&i&i&1&0 \\
0&5&2&2&i&1 \\
1&0&5&3&2&i
\end{array}\right),
\end{align*}
a $2$-level Toeplitz matrix of size $6,$ where $n_1=3$ and $n_2=2.$ By Theorem \ref{mainThm},

\begin{align*}
 U(n_1)=U(3)=\frac{1}{\sqrt{2}}\left(\begin{array}{ccc}
1 & 0 & i \\
0 & \sqrt{2} & 0 \\
1 & 0 & -i
\end{array}\right)\ \  {\rm and} \ \ U(n_2)=U(2)=\frac{1}{\sqrt{2}}\left(\begin{array}{cc}
1 &  i \\
1 &  -i
\end{array}\right).
\end{align*}

Then
\begin{align*} U_1=I_{n_2}\otimes U(n_1)=I_2\otimes U(3)=\frac{1}{\sqrt{2}}\left(\begin{array}{ccc|ccc}
1 & 0 & i & 0 & 0 & 0\\
0 & \sqrt{2} & 0 & 0 & 0 & 0\\
1 & 0 & -i & 0 & 0 & 0\\  \hline
0 & 0 & 0 & 1 & 0 & i \\
0 & 0 & 0 & 0 & \sqrt{2} & 0 \\
0 & 0 & 0 & 1 & 0 & -i
\end{array}\right),
\end{align*}
and
\begin{align*} U_2=U(n_2)\otimes I_{n_1}=U(2)\otimes I_3=\frac{1}{\sqrt{2}}\left(\begin{array}{ccc|ccc}
1 & 0 & 0 & i & 0 & 0\\
0 & 1 & 0 & 0 & i & 0\\
0 & 0 & 1 & 0 & 0 & i\\  \hline
1 & 0 & 0 & -i & 0 & 0 \\
0 & 1 & 0 & 0 & -i & 0 \\
0 & 0 & 1 & 0 & 0 & -i
\end{array}\right).
\end{align*}
So
 { \begin{align*}
 U^*_1TU_1=\frac{1}{2}\left(\begin{array}{ccc|ccc}
 3+2i & 3\sqrt{2} & 3i & 10 & 2\sqrt{2}i &0 \\
 3\sqrt{2} & 2i & \sqrt{2}i & 2\sqrt{2}i & 8 & 0 \\
 3i&\sqrt{2}i & -3+2i & 0 & 0 & 6 \\ \hline
 11+i & 2\sqrt{2} & 1+i & 3+2i & 3\sqrt{2} & 3i \\
 2\sqrt{2} & 10 & -2\sqrt{2}i & 3\sqrt{2} & 2i & \sqrt{2}i\\
 1+i & -2\sqrt{2}i& 9-i & 3i & \sqrt{2}i & -3+2i
\end{array}\right)
\end{align*}}
in which each block is symmetrized, that is the first level is symmetrized, and
 {\begin{align*}
 U^*_2U^*_1TU_1U_2=\frac{1}{4}\left(\begin{array}{ccc|ccc}
 27 + 5i & 8\sqrt{2} + 2\sqrt{2}i& 1 + 7i& - 1 + i& 2\sqrt{2} + 2\sqrt{2}i& - 1 + i\\
 8\sqrt{2} + 2\sqrt{2}i& 18 + 4i& 0& 2\sqrt{2} + 2\sqrt{2}i& 2i& 2\sqrt{2}\\
  1 + 7i& 0& 9 + 3i& - 1 + i& 2\sqrt{2}& 1 + 3i\\ \hline
  - 1 + i& 2\sqrt{2} + 2\sqrt{2}i& - 1 + i& - 15 + 3i&  4\sqrt{2} - 2\sqrt{2}i& - 1 + 5i\\
 2\sqrt{2} + 2\sqrt{2}i& 2i& 2\sqrt{2}& 4\sqrt{2} - 2\sqrt{2}i& - 18 + 4i& 4\sqrt{2}i\\
  - 1 + i& 2\sqrt{2}& 1 + 3i& - 1 + 5i& 4\sqrt{2}i& - 21 + 5i
\end{array}\right)
\end{align*}}
which is symmetric. The transition unitary matrix $U$ is given by
{\begin{align*}
 U=U_1U_2=\frac{1}{4}\left(\begin{array}{ccc|ccc}
 1& 0& i& i& 0& -1\\ 0& \sqrt{2}& 0& 0& \sqrt{2}i& 0 \\ 1 &0& -i& i& 0& 1 \\ \hline
 1& 0& i& -i& 0& 1\\0& \sqrt{2}& 0& 0& -\sqrt{2}i& 0\\1& 0& -i& -i& 0& -1
\end{array}\right).
\end{align*}}

One may use Theorem \ref{Thm3.2} as well. To construct the transition matrix, we construct $V(n_1)$ and $V(n_2)$ as the following:

\begin{align*}
 V(n_1)=V(3)=\frac{1}{\sqrt{2}}\left(\begin{array}{ccc}
1 & 0 & i \\
0 & 1+i & 0 \\
i & 0 & 1
\end{array}\right)\ \  {\rm and} \ \ V(n_2)=V(2)=\frac{1}{\sqrt{2}}\left(\begin{array}{cc}
1 &  i \\
i &  1
\end{array}\right).
\end{align*}

Then
\begin{align*} V_1=I_{n_2}\otimes V(n_1)=I_2\otimes V(3)=\frac{1}{\sqrt{2}}\left(\begin{array}{ccc|ccc}
1 & 0 & i & 0 & 0 & 0\\
0 & 1+i & 0 & 0 & 0 & 0\\
i & 0 & 1 & 0 & 0 & 0\\  \hline
0 & 0 & 0 & 1 & 0 & i \\
0 & 0 & 0 & 0 & 1+i & 0 \\
0 & 0 & 0 & i & 0 & 1
\end{array}\right),
\end{align*}
and
\begin{align*} V_2=V(n_2)\otimes I_{n_1}=V(2)\otimes I_3=\frac{1}{\sqrt{2}}\left(\begin{array}{ccc|ccc}
1 & 0 & 0 & i & 0 & 0\\
0 & 1 & 0 & 0 & i & 0\\
0 & 0 & 1 & 0 & 0 & i\\  \hline
i & 0 & 0 & 1 & 0 & 0 \\
0 & i & 0 & 0 & 1 & 0 \\
0 & 0 & i & 0 & 0 & 1
\end{array}\right).
\end{align*}
So
 { \begin{align*}
 V^*TV=\left(\begin{array}{ccc|ccc}
  -1/4-3/4i&3/2-i&7/4+1/4i&17/4-1/4i&1/2+i&3/4+1/4i \\
   3/2-i&1/2i&1/2&1/2+i&9/2&1/2 \\
   7/4+1/4i&1/2&1/4+7/4i&3/4+1/4i&  1/2&19/4+1/4i \\ \hline
   17/4-1/4i&1/2+i&3/4+1/4i&1/4-1/4i&3/2&5/4-1/4i \\
  1/2+i&9/2&1/2&3/2&3/2i&5/2+i \\
   3/4+1/4i&1/2&19/4+1/4i&5/4-1/4i&5/2+i&-1/4+13/4i
\end{array}\right)
,\end{align*}}
 where the transition unitary matrix $V=V_1V_2$ is given by
{\begin{align*}
 V=V_1V_2=\frac{1}{4}\left(\begin{array}{ccc|ccc}
 1/2& 0& 1/2i& 1/2i& 0& -1/2\\
 0& 1/2+1/2i& 0& 0& -1/2+1/2i& 0 \\
  1/2i &0& 1/2& -1/2& 0& 1/2i \\ \hline
 1/2i& 0& -1/2& 1/2& 0& 1/2i\\
 0& -1/2+1/2i& 0& 0& 1/2+1/2i& 0\\
 -1/2& 0& 1/2i& 1/2i& 0& 1
\end{array}\right).
\end{align*}}

\end{example}

\begin{example} Let \begin{align*}
 T=\left(\begin{array}{cc|cc|cc|cc}
2& 3& 4& 5& 1& 0& i& 3 \\
3 & 2&6&4&2&1&3&i \\ \hline
6&7&2&3&7&5&1&0 \\
8&6&3&2&5&7&2&1 \\ \hline
3&i&4&1&2&3&4&5 \\
i&3&1&4&3&2&6&4 \\ \hline
6&i&3&i&6&7&2&3 \\
i&6&i&3&8&6&3&2
\end{array}\right),
\end{align*}
a $3$-level Toeplitz matrix of size $8,$ where $n_1=n_2=n_3=2.$ By Theorem \ref{mainThm},
 $$U(n_1)=U(n_2)=U(n_3)=U(2)=\frac{1}{\sqrt{2}}\begin{pmatrix} 1&i \\
1 & -i \end{pmatrix}$$ and hence
\begin{align*}
 U_1=\frac{1}{\sqrt{2}} I_{n_3}\otimes I_{n_2} \otimes U(n_1)=\left(\begin{array}{cc|cc|cc|cc}
1& i& 0& 0& 0& 0& 0& 0 \\
1& -i& 0& 0& 0& 0& 0& 0 \\ \hline
0& 0& 1& i& 0& 0& 0& 0 \\
0& 0& 1& -i& 0& 0& 0& 0\\ \hline
0& 0& 0& 0& 1& i& 0& 0 \\
0& 0& 0& 0& 1& -i& 0& 0\\ \hline
0& 0& 0& 0& 0& 0&  1& i\\
0& 0& 0& 0& 0& 0&  1& -i
\end{array}\right),
\end{align*}

\begin{align*}
 U_2=\frac{1}{\sqrt{2}}I_{n_3}\otimes U(n_2) \otimes I_{n_1}=\left(\begin{array}{cc|cc|cc|cc}
1& 0& i& 0& 0& 0& 0& 0 \\
0& 1& 0& i& 0& 0& 0& 0 \\ \hline
1& 0& -i& 0& 0& 0& 0& 0 \\
0& 1& 0& -i& 0& 0& 0& 0 \\ \hline
0& 0& 0& 0& 1& 0& i& 0 \\
0& 0& 0& 0& 0& 1& 0& i\\ \hline
0& 0& 0& 0& 1& 0& -i& 0 \\
0& 0& 0& 0& 0& 1& 0& -i\\
\end{array}\right),
\end{align*}
and
\begin{align*}
 U_3= \frac{1}{\sqrt{2}}U(n_3)\otimes I_{n_2} \otimes I_{n_1}=\left(\begin{array}{cc|cc|cc|cc}
1& 0& 0& 0& i& 0& 0& 0 \\
0& 1& 0& 0& 0& i& 0& 0 \\ \hline
0& 0& 1& 0& 0& 0& i& 0 \\
0& 0& 0& 1& 0& 0& 0& i\\ \hline
1& 0& 0& 0& -i& 0& 0& 0 \\
0& 1& 0& 0& 0& -i& 0& 0\\ \hline
0& 0& 1& 0& 0& 0& -i& 0 \\
0& 0& 0& 1& 0& 0& 0& -i\\
\end{array}\right)
\end{align*}
respectively. So the transition unitary matrix is
\begin{align*}
 U=U_1U_2U_3=\frac{1}{2\sqrt{2}}\left(\begin{array}{cc|cc|cc|cc}
1& i& i& -1& i& -1& -1& -i \\
1& -i& i& 1& i& 1& -1& i \\ \hline
1& i& -i& 1& i& -1& 1& i \\
1& -i& -i& -1& i& 1& 1& -i\\ \hline
1& i& i& -1& -i& 1& 1& i \\
1& -i& i& 1& -i& -1& 1& i\\ \hline
1& i& -i& 1& -i& 1& -1& -i \\
1& -i& -i& -1& -i& -1& -1& i
\end{array}\right)
\end{align*}
and one can check that

{\begin{align*}
U^*TU=\frac{1}{8}\left(\begin{array}{cccc|cccc}
204 + 8i& 8i& 36i& 0& - 4 -4i& 4& 16 - 4i& 0\\
 8i& 8 - 4i& 0& 4 + 16i& 4& 8 + 32i& 0& 4\\
  36i& 0& -84& 0& 16 - 4i& 0& - 4 + 12i& 4\\
  0& 4 + 16i & 0 & -4i & 0 & 4 & 4 & -8i \\ \hline
  - 4 - 4i & 4 & 16 - 4i & 0 & 60 - 8i & 0 & -4i & 0\\
  4& 8 + 32i & 0 & 4 & 0 & - 48 + 4i & 0 & - 4 - 16i \\
  16 - 4i & 0 & - 4 + 12i & 4 & -4i& 0& -20 & -8i\\
  0& 4& 4& -8i & 0 & - 4 - 16i & -8i & 8 + 4i
\end{array}\right)
\end{align*}
symmetric.}

Now we are using Theorem \ref{Thm3.2} to symmetrize the same $3$-level Toeplitz matrix.

\begin{align*}
 V_1=\frac{1}{\sqrt{2}} I_{n_3}\otimes I_{n_2} \otimes V(n_1)=\left(\begin{array}{cc|cc|cc|cc}
1& i& 0& 0& 0& 0& 0& 0 \\
i& 1& 0& 0& 0& 0& 0& 0 \\ \hline
0& 0& 1& i& 0& 0& 0& 0 \\
0& 0& i& 1& 0& 0& 0& 0\\ \hline
0& 0& 0& 0& 1& i& 0& 0 \\
0& 0& 0& 0& i& 1& 0& 0\\ \hline
0& 0& 0& 0& 0& 0&  1& i\\
0& 0& 0& 0& 0& 0&  i& 1
\end{array}\right),
\end{align*}

\begin{align*}
 V_2=\frac{1}{\sqrt{2}}I_{n_3}\otimes V(n_2) \otimes I_{n_1}=\left(\begin{array}{cc|cc|cc|cc}
1& 0& i& 0& 0& 0& 0& 0 \\
0& 1& 0& i& 0& 0& 0& 0 \\ \hline
i& 0& 1& 0& 0& 0& 0& 0 \\
0& i& 0& 1& 0& 0& 0& 0 \\ \hline
0& 0& 0& 0& 1& 0& i& 0 \\
0& 0& 0& 0& 0& 1& 0& i\\ \hline
0& 0& 0& 0& i& 0& 1& 0 \\
0& 0& 0& 0& 0& i& 0& 1\\
\end{array}\right),
\end{align*}
and
\begin{align*}
 V_3= \frac{1}{\sqrt{2}}V(n_3)\otimes I_{n_2} \otimes I_{n_1}=\left(\begin{array}{cc|cc|cc|cc}
1& 0& 0& 0& i& 0& 0& 0 \\
0& 1& 0& 0& 0& i& 0& 0 \\ \hline
0& 0& 1& 0& 0& 0& i& 0 \\
0& 0& 0& 1& 0& 0& 0& i\\ \hline
i& 0& 0& 0& 1& 0& 0& 0 \\
0& i& 0& 0& 0& 1& 0& 0\\ \hline
0& 0& i& 0& 0& 0& 1& 0 \\
0& 0& 0& i& 0& 0& 0& 1\\
\end{array}\right)
\end{align*}
respectively. So the transition unitary matrix is
\begin{align*}
 V=V_1V_2V_3=\frac{1}{2\sqrt{2}}\left(\begin{array}{cc|cc|cc|cc}
1& 1& i& -1& i& -1& -1& -i \\
i& 1& -1& i& -1& i& -i& -1 \\ \hline
i& -1& 1& i& -1& -i& i& -1 \\
-1& i& i& 1& -i& -1& -1& i\\ \hline
1& -1& -1& -i& 1& i& i& -1 \\
-1& i& -i& -1& i& 1& -1& i\\ \hline
-1& -i& i& -1& i& -1& 1& i \\
-i& -1& -1& i& -1& i& i& 1
\end{array}\right)
\end{align*}
and one can check that

{\begin{align*}
V^*TV=\left(\begin{array}{cccc|cccc}
30-18i&34-6i&38-10i&54+14i&14-22i&6+2i&34+2i&18+2i\\
34-6i&22-18i&54+14i&38-2i&6+2i&14-14i&18+2i&34+2i\\
38-10i&54+14i&10+2i&22+14i&34+2i&18+2i&18+4i& 2+6i\\
54+14i&38-2i &22+14i &2+2i &18+2i &34+2i & 2+6i &18+22i \\ \hline
14-22i &6+2i & 34+2i & 18+2i & 2+2i & 14-10i & 42+2i & 50-14i\\
6+2i&14-14i &18+2i&34+2i &14-10i &10+2i &50-14i & 42+10i \\
34+2i &18+2i& 18+14i &2+6i & 42+2i& 50-14i& 22+14i & 26+2i\\
18+2i& 34+2i& 2+6i& 18+22i &50-14i & 42+10i & 26+2i & 30+14i
\end{array}\right)
\end{align*}}
 is symmetric and note that the resulting symmetric matrices are not necessarily the same.

\end{example}

\section{Symmetric matrices that are unitarily similar to Toeplitz matrices}

Let $T_n$ be an $n \times n$ $p$-level Toeplitz matrix. According to Theorem \ref{mainThm}, there exists a unitary matrix $U,$ such that $UT_nU^*$ is a symmetric matrix. However, the converse is not true, i.e., not every complex symmetric matrix is unitarily similar to a (multilevel) Toeplitz matrix (see Section 5 and Section 6 in \cite{TamLiu}). Denote $\mathcal{S}_{2^p}$ the set of all $2^p\times 2^p$ complex symmetric matrices. In this section, we provide the necessary and sufficient condition under which a matrix in $\mathcal{S}_{2^p}$ is similar to a $2^p\times 2^p$ $p$-level Toeplitz matrix under the unitary transformation given in Section 3.

Let $S\in \mathcal{S}_{2^p}.$ Let $q$ be a positive integer less than or equal to $p.$ Then $S$ can be written as
$$S=\begin{pmatrix}S_{11}& S_{12} & \ldots & S_{1r}  \\ S_{21}& S_{22} & \ldots & S_{2r} \\ \vdots&\vdots&\ddots&\vdots \\S_{r1}& S_{r2} & \ldots & S_{rr}\end{pmatrix}$$
where $r=2^{p-q}$ and each $S_{ij}$ is a $2^q \times 2^q$ matrix for $i,j=1,2,\ldots,r.$ For $i,j=1,2,\ldots, r,$ each $S_{ij}$ is called a $q$-level block of $S$ and $S$ is said to have $q$-level constant anti-diagonals if each $S_{ij}$ has a constant anti-diagonal. $S$ having constant anti-diagonals at each level means that $S$ has $q$-level constant anti-diagonals for all $q=1,2,\ldots,n,$

\begin{example} Let $p=2,$ and $$S=\begin{pmatrix}s_{11} & s_{12}  & s_{13} & s_{14} \\ s_{21} & s_{22}  & s_{23} & s_{24} \\ s_{31} & s_{32}  & s_{33} & s_{34} \\ s_{41} & s_{42}  & s_{43} & s_{44} \\\end{pmatrix}\in \mathcal{S}_{2^2}.$$ $S$ having $1$-level constant anti-diagonals means that
\begin{equation}\label{1level}s_{14}=s_{23}, \ \ s_{32}=s_{41},\ \  s_{12}=s_{21}, \ \ {\rm and}\ \ s_{34}=s_{43}.\end{equation}
$S$ having $2$-level constant anti-diagonals means that
\begin{equation}\label{2level}s_{14}=s_{23}= s_{32}=s_{41}.\end{equation}
$S$ having constant diagonals at each level means both \eqref{1level} and \eqref{2level}.
\end{example}

Given a positive integer $p.$ Let $n_1=n_2=\ldots=n_{p-1}=n_p=2.$  Then let
$$\mathcal{U}=\frac{1}{\sqrt{2}}\begin{pmatrix}1 & i \\ 1 & -i  \end{pmatrix}$$
and
$$U_k=I_{n_p}\otimes I_{n_{p-1}} \otimes \ldots \otimes I_{n_{k+1}}\otimes \mathcal{U} \otimes I_{n_{k-1}}\otimes \ldots \otimes I_{n_2} \otimes I_{n_1}$$ for $k=1,2,\ldots, p.$ Denote
\begin{equation}\label{U} U^{(p)}=\prod_{j=1}^p U_j.\end{equation}

\begin{lemma} \label{lm5.1}

Let $T^{(p)}$ be a $2^p\times 2^p$ $p$-level Toeplitz matrix. Let $U$ be the unitary matrix defined by \eqref{U}. Then the complex symmetric matrix $$ S^{(p)}=(U^{(p)})^*T^{(p)}U^{(p)}$$ has constant anti-diagonals at each level.

\end{lemma}

\begin{proof}

We use induction on $p.$

Base case: When $p=1,$ $S^{(1)}$ has a constant anti-diagonal due to the symmetry of $S^{(1)}.$

\medskip

Inductive assumption: Suppose it is true for $p=m.$ That is the $2^m \times 2^m $ complex symmetric matrix $S^{(m)}=(U^{(m)})^*T^{(m)}U^{(m)}$ has constant anti-diagonals on each level.

\medskip

Inductive step: We need to show for $p=m+1,$ the $2^{m+1} \times 2^{m+1}$ complex symmetric matrix $S^{m+1}=(U^{(m+1)})^*T^{(m+1)}U^{(m+1)}$ has constant anti-diagonals on each level.

\medskip

First note that $$T^{(m+1)}=\begin{pmatrix}T_0^{(m)}& T_{-1}^{(m)}\\ T_1^{(m)} & T_0^{(m)}\end{pmatrix}$$
where $T_0^{(m)}, T_{-1}^{(m)}$ and $T_1^{(m)}$ are $2^m \times 2^m$ $m$-level Toeplitz matrices. According to the inductive assumption, there exists a unitary matrix $U^{(m)}$ such that all $$S_0^{(m)}=(U^{(m)})^*T_0^{(m)}U^{(m)}, \ \ S_{-1}^{(m)}=(U^{(m)})^*T_{-1}^{(m)}U^{(m)}\ \  {\rm \ and} \ \ S_{1}^{(m)}=(U^{(m)})^*T_{1}^{(m)}U^{(m)}$$ are complex symmetric matrices with constant anti-diagonals at each level. That is,
\[ S^{(m)}=\begin{pmatrix}U^{(m)} & \\ & U^{(m)} \end{pmatrix}^*\begin{pmatrix}T_0^{(m)}& T_{-1}^{(m)}\\ T_1^{(m)} & T_0^{(m)}\end{pmatrix} \begin{pmatrix}U^{(m)} & \\ & U^{(m)} \end{pmatrix}=\begin{pmatrix}S_0^{(m)} & S_{-1}^{(m)}\\ S_1^{(m)}& S_0^{(m)} \end{pmatrix}.  \]
Let $$U_{m+1}=\mathcal{U}\otimes(\otimes_{k=1}^m I_2)=\frac{1}{\sqrt{2}}\begin{pmatrix}I^{(m)} & iI^{(m)} \\I^{(m)} & -iI^{(m)}\end{pmatrix}$$ where $I^{(m)}$ is the $2^m \times 2^m$ identity matrix. Then
\[ S^{(m+1)}=U_{m+1}^* S^{(m)}U_{m+1}=\frac{1}{2}\begin{pmatrix}2S_{0}^{(m)} + S_{-1}^{(m)} +S_{1}^{(m)} & -iS_{-1}^{(m)}+iS_{1}^{(m)} \\ -iS_{-1}^{(m)}+iS_{1}^{(m)} & 2S_{0}^{(m)} - S_{-1}^{(m)} -S_{1}^{(m)} \end{pmatrix}\] which has constant anti-diagonals at each level.

\end{proof}

\begin{lemma}\label{lm5.2} Let $S$ be a $2^p\times 2^p$ complex symmetric matrix. If $S$ has constant anti diagonals for each level, then $U^{(p)}S(U^{(p)})^*$ is a $p$-level Toeplitz matrix.

\end{lemma}
\begin{proof}

We use induction on $p.$

\medskip

Base case: For $p=1,$ Let $S^{(1)}=\begin{pmatrix} s_{11} & s_{12} \\ s_{12}& s_{22} \end{pmatrix}.$ Then
\begin{eqnarray} U^{(1)}S^{(1)}(U^{(1)})^* &=& \frac{1}{2} \begin{pmatrix} 1&i \\ 1& -i \end{pmatrix} \begin{pmatrix} s_{11} & s_{12} \\ s_{12}& s_{22} \end{pmatrix} \begin{pmatrix} 1&1 \\-i&i \end{pmatrix} \nonumber \\
                                     &=& \frac{1}{2} \begin{pmatrix} s_{11}+s_{22} & s_{11}-s_{22}+2is_{12} \\ s_{11}+s_{22}-2is_{12} & s_{11}+s_{22} \nonumber   \end{pmatrix}
\end{eqnarray}a $1$-level Toeplitz matrix.
\medskip

Inductive assumption: Suppose it is true for $p=m.$ That is, for a $2^m \times 2^m$ complex symmetric matrix $S^{(m)}$ with constant anti-diagonals at each level, $U^{(m)}S^{(m)}(U^{(m)})^*$ is an $m$-level Toeplitz matrix.

\medskip

Inductive step: We need to show for $p=m+1,$  if $S^{(m+1)}$ is a $2^{m+1}\times 2^{m+1}$ complex symmetric matrix with constant anti-diagonals at each level, then $U^{(m+1)}S^{(m+1)}(U^{(m+1)})^*$ is a $(m+1)$-level Toeplitz matrix.

Note that $$S^{(m+1)}=\begin{pmatrix}S_{11}^{(m)}& S_{12}^{(m)}\\ S_{12}^{(m)} & S_{22}^{(m)}\end{pmatrix}$$
where $S_{11}^{(m)}, S_{12}^{(m)}$ and $S_{22}^{(m)}$ are $2^m \times 2^m$ complex symmetric matrices with constant anti-diagonals at each level. According to the inductive assumption, there exists a unitary matrix $U^{(m)}$ such that all $$T_{11}^{(m)}=(U^{(m)})S_{11}^{(m)}(U^{(m)})^*, \ \ T_{12}^{(m)}=(U^{(m)})S_{12}^{(m)}(U^{(m)})^*\ \  {\rm \ and} \ \ T_{22}^{(m)}=(U^{(m)})S_{22}^{(m)}(U^{(m)})^*$$ are $m$-level Toeplitz matrices. That is,
\[ T=\begin{pmatrix}U^{(m)} & \\ & U^{(m)} \end{pmatrix}\begin{pmatrix}S_{11}^{(m)}& S_{12}^{(m)}\\ S_{12}^{(m)} & S_{22}^{(m)}\end{pmatrix} \begin{pmatrix}U^{(m)} &  \\ & U^{(m)}\end{pmatrix}^*=\begin{pmatrix}T_{11}^{(m)} & T_{12}^{(m)}\\ T_{12}^{(m)}& T_{22}^{(m)} \end{pmatrix}.\]
We define $U_{m+1}$ as  $$U_{m+1}=(\otimes_{k=1}^m I_2) \otimes \mathcal{U}=\frac{1}{\sqrt{2}}\begin{pmatrix}I^{(m)} & iI^{(m)} \\I^{(m)} & -iI^{(m)}\end{pmatrix},$$ where $I^{(m)}$ is the $2^m \times 2^m$ identity matrix.  Then
\[ T^{(m+1)}=U_{m+1} T(U_{m+1})^*=\frac{1}{2}\begin{pmatrix} T_{11}^{(m)} +T_{22}^{(m)} & T_{11}^{(m)}-T_{22}^{(m)}+2iT_{12}^{(m)} \\  T_{11}^{(m)}+T_{22}^{(m)}-2iT_{12}^{(m)} & T_{11}^{(m)} +T_{22}^{(m)} \end{pmatrix}\] is a $(m+1)$-level Toeplitz matrix.

\end{proof}

Combine Lemma \ref{lm5.1} and Lemma \ref{lm5.2}, we have the following theorem.
\begin{theorem} Let $S$ be a $2^p \times 2^p$ complex symmetric matrix. There exists a $2^p \times 2^p$ $p$-level Toeplitz matrix $T^{(p)}$ such that $$S=(U^{(p)})^*T^{(p)}U^{(p)}$$ if and only if $S$ has constant anti-diagonals at each level.
\end{theorem}

\noindent{\bf Acknowledgements}

We would like to thank Drs.Banani Dhar and Sarita Nemani for insightful discussions.

\bibliographystyle{plain}

\end{document}